\documentclass{amsart}

\usepackage{amsmath,amssymb,amsthm,verbatim}
\usepackage{hyperref}
\usepackage{color}

\newtheorem{theorem}[subsection]{Theorem}
\newtheorem{lemma}[subsection]{Lemma}
\newtheorem{cor}[subsection]{Corollary}
\newtheorem{prop}[subsection]{Proposition}

\newtheorem{property}[subsection]{Property}

\newtheorem*{assertion*}{Assertion}

\theoremstyle{definition}
\newtheorem{definition}[subsection]{Definition}
\newtheorem{remark}[subsection]{Remark}

\newcommand{\dist}{\mathrm{dist}}

\newcommand{\Ric}{\mathrm{Ric}}

\newcommand{\R}{\mathbb{R}}

\newcommand{\N}{\mathbb{N}}

\newcommand{\del}{\partial}

\newcommand{\cS}{\mathcal{S}}

\newcommand{\loc}{loc}

\DeclareMathOperator{\Div}{div}

\begin{document}

\title[Neumann eigenvalue estimates and elliptic regularity]{Sharp Neumann eigenvalue estimates and $C^2$ elliptic regularity in non-obtuse polyhedral domains}
\author{Nick Edelen}
\address{Department of Mathematics, University of Notre Dame, Notre Dame, IN 46556 USA}
\email{nedelen@nd.edu}
\author{Chao Li}
\address{Courant Institute, New York University, 251 Mercer St, New York, NY 10012, USA}
\email{chaoli@nyu.edu}

\begin{abstract}
    For integers $n\ge 1$, consider assertions:
    
    \noindent $\mathbf{(P_n)}$: Let $\Omega \subset S^{n-1}$ be a spherical domain enclosed by totally geodesic $S^n$'s with non-obtuse dihedral angles. Then in $[0,2(n+1)]$, its Neumann spectrum can only take values among $\{0,n,2(n+1)\}$. Moreover, $n$ is a Neumann eigenvalue if and only if the corresponding eigenfunction is the restriction of a linear function in $\R^{n+1}$, while $2(n+2)$ is a Neumann eigenvalue if and only if the corresponding eigenfunction is restriction of a quadratic polynomial in $\R^{n+1}$.
    

    \noindent $\mathbf{(Q_n)}$: A weak solution $u$ to $\Delta u = f$ with the Neumann boundary condition, with $f$ H\"older continuous, in a conical polyhedral domain $\Omega$ in $\R^n$ with non-obtuse dihedral angles, is in $C^{2,\alpha}_{\loc}(\overline \Omega)$.

    We prove the implications 
    \[\mathbf{(Q_n)} \Rightarrow  \mathbf{(P_n)},\qquad \mathbf{(P_n)}\Rightarrow \mathbf{(Q_{n+1})}.\]
    Consequently, both assertions hold in all dimensions. These give the optimal Neumann eigenvalue lower bound and $C^2$ elliptic regularity in non-obtuse Riemannian polyhedral domains.

\end{abstract}

\maketitle

\section{Introduction}

We develop a $C^2$ regularity theory for solutions to second-order elliptic equations in Riemannian polyhedral domains with the Neumann boundary conditions. Beyond their intrinsic interest, results of this type have had far-reaching applications in comparison geometry of curvature conditions, including a wave of recent progress on polyhedral rigidity theorems proposed by Gromov \cite{GromovDiracbilliards}. We start by definition the domains under consideration.

\begin{definition}[Riemannian polyhedron]\label{defi. riemannian polyhedorn}
    Let $M$ be a Hausdorff, second countable topological space. We say that $(M,g)$ is an $n$-dimensional smooth Riemannian polyhedron, if for every point $x\in M$, there exist a closed $n$-dimensional Riemannian manifold $(\hat M^n,\hat g)$ without boundary, smooth domains-with-boundary
    \begin{align*}
        M_1,\dots,M_k\subset \hat M, \qquad k\in \N,
    \end{align*}
    whose boundary hypersurfaces meet transversely, and an open set $U\subset \hat{M}$, such that $(M,g)$ is locally isometric to
    \begin{align*}
        (M_1\cap M_2\cap\dots\cap M_k\cap U, \hat{g}|_{M_1\cap M_2\cap\dots\cap M_k\cap U})).
    \end{align*}
    At a point $p\in \partial M\cap \partial M_i\cap \partial M_j$, the dihedral angle $\theta\in (0,\pi)$ at $p$ between $M_i$ and $M_j$ is defined via $\cos\theta = -\nu_i\cdot \nu_j$, where $\nu_j, \nu_j$ are the outward unit normal of $M_i, M_j$, respectively.
\end{definition}

Standard examples include \textit{Euclidean polyhedral cones} as well as \textit{spherical polyhedral domains}: we say $\Omega\subset \R^n$ is an Euclidean polyhedral cone if it is the intersection of finitely many closed half spaces and is dilation invariant; $\Omega\subset S^n$ is a spherical polyhedral domain, if it is the intersection of finitely many closed hemispheres in $S^n$.

In local coordinates, consider the second order divergence-form elliptic Neumann problem in $\Omega$:
\begin{equation}\label{eq. genera. elliptic.equation}
    \begin{cases}
    Lu = -\nabla_i (a^{ij} \nabla_j u) + b^i \nabla_i u + cu = f &\quad \text{ in }\mathring\Omega,\\
    a^{ij} \nu_i \nabla_j u = 0 &\quad \text{ on }\partial \Omega.
    \end{cases}
\end{equation}

The central question is to determine the optimal regularity of a weak solution $u$ under natural assumptions on the coefficients. Problems of this type have been studied extensively. Foundational contributions include Kondrat'ev's treatment of conical domains by means of weighted estimates \cite{KondratWeightedEstimates}, as well as the comprehensive monographs of Dauge \cite{DaugeBook} and Maz'ya--Rossmann \cite{MazyaRossmanBook}.

The guiding principle is that standard boundary regularity controls the smoothness of $u$ in the interiors of the faces of $\Omega$, whereas regularity near an edge or vertex is governed by the frozen boundary-value problem on the corresponding tangent cone. To illustrate this principle, let $\Omega\subset\R^{n}$ be a polyhedral cone and set $\Sigma = \Omega\cap S^{n-1}$  so that $\Sigma$ is a spherical polyhedral domain. A harmonic function $u$
on $\Omega$ satisfying the homogeneous Neumann boundary condition and having locally finite Dirichlet energy admits a Fourier expansion of the form
\[
    u(r,\theta)
    =
    a_0+\sum_{j\geq 1}a_j r^{\gamma_j}\phi_j(\theta),
\]
where $r$ is the radial variable, $\theta\in S^{n-1}$, and $\phi_j$ is a Neumann eigenfunction of $-\Delta$ on $\Sigma$ with eigenvalue $\mu_j$. The homogeneity exponent $\gamma_j$ is determined by
\begin{equation}\label{eq. expansion.rate.vs.eigenvalue}
    \gamma_j(\gamma_j+n-2)=\mu_j,\qquad j\geq 1.
\end{equation}
Thus, lower bounds for $\mu_j$ yield larger homogeneity exponents $\gamma_j$, and hence faster radial decay and improved regularity at the vertex. This principle has been used, through weighted estimates, to obtain regularity results on polyhedral domains; see, among others, \cite{DaugeNeumann.curvilinear,MazyaRossmannSchauder,Mazya2009boundednessofgradient}.

As observed by Maz\'ya \cite{Mazya2009boundednessofgradient}, a natural geometric assumption to guarantee boundedness of $\nabla u$ is \textit{convexity} of tangent cones of $\Omega$: indeed, by the standard Bochner formula, whenever the tangent cone $\Omega\subset \R^{n}$ is convex, its spherical link $\Omega\cap S^{n-1}$ satisfies $\mu_1\ge n-1$ (see, e.g. \cite{Mazya2009boundednessofgradient,EscobarEigenvalueEstimate}), which translates to $\lambda_1 \ge 1$ through \eqref{eq. expansion.rate.vs.eigenvalue} and gives the desired $C^{0,1}$ regularity.

In this paper, we seek an analogous geometric condition guaranteeing $C^2$ regularity for weak solutions of \eqref{eq. genera. elliptic.equation}. Such estimates are indispensable in geometric comparison and rigidity arguments \cite{LiDihedralRigidityPrisms,LiDihedralHyperbolic}, where curvature quantities involve second derivatives. At the level of the model cone, the critical spectral threshold is
\[
    \mu\geq 2n
    \quad\Longleftrightarrow\quad
    \gamma\geq 2.
\]
Consider, for example, the planar wedge $W_\omega = \{(r,\theta):r>0,\ 0<\theta<\omega\} \subset\R^2$.
The first positive Neumann eigenvalue of its spherical link is $\mu_1=\left(\frac{\pi}{\omega}\right)^2$. Consequently,
\[
    \omega\leq\frac{\pi}{2}
    \quad\Longleftrightarrow\quad
    \mu_1\geq 4
    \quad\Longleftrightarrow\quad
    \gamma_1\geq 2,
\]
which is precisely the threshold suggested by $C^2$ regularity.

This model indicates that a necessary local condition for a general $C^2$ regularity theory is that the cone be \emph{non-obtuse}: every dihedral angle should be at most $\pi/2$. Indeed, the tangent cone along a codimension-two singular stratum has the form $W_\omega\times\R^\ell$, and local $C^2$ regularity requires the opening angle $\omega$ of the wedge to be non-obtuse. This mechanism was previously implemented in the special setting in which $\Omega$ is locally modeled on $W\times\R^k\times[0,\infty)\times\R^{\ell-k}$ \cite{LiDihedralRigidityPrisms}; after even reflections across the flat boundary factor, the problem reduces to a product of a wedge with a Euclidean space.

The main result of this paper confirms this heuristic in all dimensions. In fact, we prove that, this sharp eigenvalue estimate for spherical polyhedral domains is intimately connected with the sharp regularity of elliptic equations. Precisely, we consider two dimension-dependent assertions $P_n$ and $Q_n$ as follows.

\begin{assertion*}[$P_n$]
    Let $\Omega \subset S^n$ be a spherical polyhedral domain with non-obtuse dihedral angles, and let $0 = \mu_0 < \mu_1 \leq \mu_2  \cdots$ be its Neumann spectrum.  Then $\{ \mu_i \}_{i} \cap [0, 2(n+1)] = \{ 0, n, 2(n+1) \}$.  Moreover, $n$ is a Neumann eigenvalue if and only if the corresponding eigenfunction is the restriction of a linear function, while $2(n+2)$ is a Neumann eigenvalue if and only if the corresponding eigenfunction is restriction of a quadratic polynomial.
\end{assertion*}

\begin{assertion*}[$Q_n$]
   Let $\Omega\subset\R^n$ be a polyhedral cone with non-obtuse dihedral
    angles. Suppose that $u\in W^{1,2}(\Omega \cap B_1)$ is a weak solution of
    \[
        \Delta u=f\quad\text{in }\mathring{\Omega},
        \qquad
        D_\nu u=0\quad\text{on }\partial\Omega,
    \]
    where $f \in C^\alpha(\Omega \cap B_1)$. Then there is a $\beta(\Omega, \alpha)$ so that $u \in C^{2,\beta}(\Omega \cap B_{1/2})$, and moreover
    \[
        |u|_{C^{2,\beta}(\Omega \cap B_{1/2})} \leq c(\Omega, \alpha)( ||u||_{L^2(\Omega \cap B_1)} + |f|_{C^\alpha(\Omega\cap B_1)})
    \]
\end{assertion*}

Our main theorem is:

\begin{theorem}\label{theorem. main. Pn and Qn}
    \[(P_n)\Rightarrow (Q_{n+1}) ,\qquad (Q_n)\Rightarrow (P_n).\]
    Consequently, both assertions $P_n$ and $Q_n$ hold in every dimension.
\end{theorem}

An important special case where Assertion $(P_n)$ applies is when $\Omega\subset S^n$ is a spherical simplicial complex - that is, it is enclosed by precisely $(n+1)$ totally geodesic hypersurfaces in general positions. In this case, the restriction of any linear function in $\R^{n+1}$ cannot satisfy the Neumann boundary condition on $\partial \Omega$, hence $n$ is not in the Neumann spectrum. Consequently, we have:

\begin{cor}\label{corollary.eigenvalue.simplicial complex}
    Let $\Omega\subset S^n$ be a non-obtuse spherical simplicial complex. Then its first nonzero Neumann eigenvalue satisfies
    \[\mu_1(\Omega)\ge 2(n+1). \]
    Equality holds if and only if the corresponding eigenfunction is the restriction of a harmonic quadratic polynomial in $\R^{n+1}$.
\end{cor}

\subsection{Proof strategies}

The implication $(Q_n)\Rightarrow (P_n)$ is a sharp eigenvalue estimate on spherical domains. The key geometric input is a pointwise, tensor-valued equation for the traceless Hessian $T = \nabla^2 u - \tfrac{\Delta u}{n} g$ of any Laplacian eigenfunction $u$ on $S^n$:
\[\Delta T = (2n-\lambda)T, \]
where $\lambda$ is the eigenvalue (see Proposition \ref{prop. geometric equations for T}). The regularity supplied
by $(Q_n)$ allows us to integrate this to a Reilly-type formula for $|\nabla T|^2$ on a polyhedral domain. We emphasize that this argument, in the smooth setting, is not entirely new: it goes back to a classical work by Simon \cite{Simon1978} and was recently sharpened by Guan-Guo \cite{GuanGuo}, in connection with the study of Einstein manifolds. Our main new contribution here is a sharp characterization of the equality case \footnote{We emphasize that the equality characterization was also obtained, in a weaker form, in \cite{GuanGuo}.}, which is crucial to the regularity theorem.

The implication $(P_n)\Rightarrow(Q_{n+1})$ is a local regularity theorem. The spherical link of an $(n+1)$-dimensional non-obtuse polyhedral cone is an $n$-dimensional spherical polyhedral domain to which $(P_n)$ applies. The assertion $(P_n)$ rules out nonzero homogeneous Neumann solutions of degrees strictly between $1$ and $2$ and identifies the degree-two solutions with quadratic polynomials. This spectral information serves as an integrability condition: after subtracting the appropriate affine and quadratic approximations, it yields an excess-decay estimate and, ultimately, the desired $C^{2,\alpha}$ regularity.

\subsection{The use of AI}
We used ChatGPT Pro 5.6 Sol for reference searches, and to find the best constant in the algebraic Lemma \ref{lemma.algebraic inequality} as well as a clean presentation of its proof. Other parts of this article is solely at the responsibility of the authors. This article does not contain any AI-generated text.

\subsection{Acknowledgment}
The $C^2$ regularity problem stemmed naturally out of the the Ph.D. thesis of C.L. Since then, C.L. has benefited a lot from helpful conversations with various people including Sasha Logunov, Jonathan Zhu, Guofang Wei, Fanghua Lin, Guido De Philippis, Feng Luo and Jiakun Liu, and he wishes to thank them all. C.L. is supported by NSF grant DMS-2303624 and a Sloan Fellowship. N.E. is supported by NSF grant DMS-2506700 and a Simons Foundation travel award.  

\section{The implication {$(Q_n)\Rightarrow(P_n)$}}

In this section we prove the implication $(Q_n)\Rightarrow(P_n)$.  This is a sharp eigenvalue estimate assuming $C^2$ regularity of the corresponding eigenfunctions. The geometric argument in this section is a local, pointwise version of in the integral estimates in the recent work of Guan-Guo \cite[Lemma 3.1]{GuanGuo}, which is a refinement of a classical result due to Simon \cite{Simon1978} for eigenvalue estimates on Einstein manifolds. We nevertheless include all details, highlighting the use of the regularity assumption, as well as our contribution on the rigidity analysis, which plays a key role in its applications to regularity theorems.

\begin{theorem}\label{theorem.eigenvalue estimate}
    Assume Assertion $(Q_n)$ holds. Let $\Omega\subset S^n$  be a spherical polyhedral domain with non-obtuse dihedral angles. Then any Laplacian eigenvalue with the Neumann boundary condition cannot lie in $(n, 2(n+1))$. Moreover, any eigenfunction corresponding to the eigenvalue $2(n+1)$ is the restriction of a quadratic harmonic polynomial in $\R^{n+1}$ from the embedding $\Omega \subset S^n \hookrightarrow \R^{n+1}$.
\end{theorem}

The rest of this section is devoted to proving Theorem \ref{theorem.eigenvalue estimate}. 

    Denote by $g$ the round metric on $S^n$ and let $u\in W^{1,2}(\Omega)$ be a nonconstant eigenfunction so that
    \[\Delta u = -\lambda u, \qquad D_\nu u = 0.\]
    Consider the traceless Hessian of $u$:
    \[T= \nabla^2 u - \frac{\Delta u}{n} g= \nabla^2 u  + \frac{\lambda}{n} u g. \]

    We first deduce the necessary regularity property of $u$.

    \begin{prop}
        $u\in C^{2,\alpha}(\overline \Omega)$ for some $\alpha>0$.
    \end{prop}

    \begin{proof}
        Since $\Omega$ is convex, we have that $u\in C^{1,\beta}(\overline \Omega)$ for some $\beta>0$. Take $p\in \partial \Omega$. If $p$ lies in $\mathring \Omega$ or the smooth part of $\partial \Omega$, then $u$ is locally $C^{2,\alpha}$ in a neighborhood of $p$ by standard regularity theory. Suppose now that $p$ lies on the singular part of $\partial \Omega$. Under the standard stereographic projection $\Phi$ centered at $p$, there exists a neighborhood $U$ of $p$, a non-obtuse Euclidean polyhedral cone $\Omega'$, so that
        \[\Phi: U\to \Omega' \text{ is a diffeomorphism}, \qquad g= \frac{4}{(1+|x|^2)^2} (dx)^2.\]
        Thus, in these coordinates, the equation $\Delta_g u = -\lambda u$ becomes
        \[\partial_i (A(x) \partial_i u) =   f \quad  \Rightarrow \quad \Delta u = A^{-1}f - A^{-1}\partial_i A \partial_i u,\]
        for some smooth $A$ and $f\in C^{0,\beta}(\overline{\Omega'})$. Thus, $(Q_n)$ implies that $u\in C^{2,\alpha}(\overline {\Omega'})$, for some $\alpha>0$. The conclusion follows from compactness of $\overline \Omega$.
    \end{proof}

    \begin{remark}
        A conceptually different proof of the desired regularity of $u$ is the following. Fix $p$ on the boundary of $\Omega$. Extend $u$ homogeneously to a harmonic function $\bar u$ in the cone over $\Omega$ in $\R^{n+1}$. After re-centering at $p$, the function $\bar u$ is harmonic in a polyhedral cone $ \Omega'$ in $\R^{n+1}$ and has at least one translation symmetry. Then Theorem \ref{thm:P-implies-Q}, applied in $\Omega'$, implies that $\bar u$ is $C^{2,\alpha}$ in $\overline{\Omega'}\cap B_{\varepsilon}$ for some $\varepsilon>0$.
    \end{remark}

    Thus $T\in C^{0,\alpha}(\overline \Omega)$. In fact, the minimum necessary regularity for this argument is that $T\in W^{2,\infty}(\Omega)$. By the Bochner formula,
    \begin{align*}
        \frac12 \Delta |\nabla u|^2 &= |\nabla^2 u|^2 + \langle \nabla u, \nabla \Delta u\rangle  + (n-1)|\nabla u|^2\\
        &= |\nabla^2 u|^2 - (\lambda -n+1)|\nabla u|^2.
    \end{align*}
    Observe that, on the smooth part of $\partial \Omega$,
    \[\partial_\nu |\nabla u|^2 = 2\nabla^2 u(\nu, \nabla u),\]
    and 
    \[0 = \partial_{\nabla u}\langle \nabla u, \nu\rangle = \nabla^2 u(\nabla u, \nu) - \langle \nabla u, \nabla_{\nabla u}\nu\rangle = \nabla^2 u(\nabla u, \nu). \]
    Here we have used that $\nabla_{\nabla u}\nu = 0$ since $\partial \Omega$ is totally geodesic.

    Because of $(Q_n)$, we know that $|\nabla u|^2\in C^{0,1}(\overline{\Omega})$. Round up the corner of $\Omega$ by $\Omega_\varepsilon$ and integrate the Bochner formula, we hence get that 
    \[\int_{\Omega_\varepsilon} |\nabla^2 u|^2 - (\lambda-n+1) |\nabla u|^2 = \int_{\partial \Omega_\varepsilon \setminus \partial \Omega} \partial_\nu |\nabla u|^2. \]
    As $\varepsilon\to 0$, we have that $|\partial \Omega_\varepsilon \setminus \partial \Omega| \to 0$, and $\partial_\nu |\nabla u|^2$ is uniformly bounded. Consequently, we obtain that
    \[\int_\Omega |\nabla^2 u|^2 = (\lambda-n+1) \int_\Omega |\nabla u|^2 = \lambda(\lambda-n+1)\int_\Omega u^2. \]

    Since $T$ is traceless, we have that 
    \begin{equation}\label{eq. integral T and u}
        \int_\Omega |T|^2  = \int_\Omega |\nabla^2 u|^2 - \frac 1n \int_\Omega (\Delta u)^2 = \frac{n-1}{n}\lambda(\lambda- n)\int_\Omega u^2.
    \end{equation}
    So we must have that $\lambda \ge n$. To prove Theorem \ref{theorem.eigenvalue estimate}, we assume that $\lambda>n$.

\begin{prop}\label{prop. geometric equations for T}
    We have:
    \begin{equation}\label{eq.div T}
            \Div T = -\frac{n-1}{n} (\lambda -n) \nabla u.
        \end{equation}
    \begin{equation}\label{eq.Delta T}
            \Delta T = (2n-\lambda) T.
        \end{equation}
   
\end{prop}

\begin{proof}
    Take coordinates $(x^1,\cdots,x^n)$ normal at a point. We use subscripts to denote a covariant derivative. We compute, for any $j$:
    \begin{align*}
        \nabla_k (\nabla^2 u)_{kj} &= \nabla_k \nabla_k u_j = \nabla_k \nabla_j u_k\\
        &= \nabla_j \nabla_k u_k + R_{kjkl} u_l\\
        &= \nabla_j (\Delta u) + \Ric_j^l u_l.
    \end{align*}
    This gives that 
    \[\Div \nabla^2 u = d (\Delta u) + \Ric (\nabla u, \cdot) = -\lambda du + (n-1)du.\]
    For the second assertion, we use the commutator identity on the sphere and deduce:
    \begin{align*}
        \Delta (\nabla^2 u) &= \nabla^2 (\Delta u) + 2n \nabla^2 u - 2(\Delta u) g\\
        &=(2n-\lambda) \nabla^2 u + 2\lambda ug. \qedhere
    \end{align*}
\end{proof}
    
    Using \eqref{eq.div T}, we have that
    \[\int_\Omega |\Div T|^2 = \left(\frac{n-1}{n}(\lambda-n)\right)^2 \int_\Omega |\nabla u|^2. \]

    Next, we would like to integrate \eqref{eq.Delta T}. This is the key place where we needed the $W^{2,\infty}$ regularity of $T$.
    
    \begin{lemma}\label{lemma. integrability of |nabla T| square}
        \begin{equation}\label{eq. integral of T and nabla T}
        0= \int_\Omega (2n-\lambda) |T|^2  + |\nabla T|^2.
    \end{equation}
    
    \end{lemma}
    \begin{proof}
        From \eqref{eq.Delta T}, we first obtain that, on the smooth part of $\Omega$, 
        \begin{equation}\label{eq.equation for Delta |T| square}
            \frac12 \Delta |T|^2 = |\nabla T|^2 + (2n-\lambda)|T|^2.
        \end{equation}
        Denote by $\cS$ the singular strata $\partial \Omega$. It follows that the codimension of $\cS$ is at least $2$. For $\varepsilon>0$, set $\rho = \dist (\cS,\cdot)$. Consider the standard log-cut off Lipschitz function
        \[\varphi = \begin{cases}
            0 &\qquad \rho\le \varepsilon^2\\
             \frac{-\log \rho + 2 \log \varepsilon}{\log \varepsilon} &\qquad \varepsilon^2 \le \rho\le \varepsilon\\
             1 &\qquad \rho> \varepsilon.
        \end{cases}\]
        We have that $\int |\nabla \varphi|^2 \le C |\log \varepsilon|^{-1}$ for $C$ independent of $\varepsilon$. By \eqref{eq.Delta T}, we have
        \begin{multline*}
            (\lambda-2n)\int |T|^2 \varphi^2 = -\int  \langle \Delta T, \varphi^2 T\rangle =\int \langle \nabla T, \nabla (\varphi^2 T)\\
            =\int2\phi \langle \nabla T, T\rangle \cdot \nabla \varphi + \langle \nabla T, \varphi^2\nabla T\rangle .
        \end{multline*}
        Thus, Young's inequality implies that
        \[\frac12 \int \varphi^2 |\nabla T|^2 \le |2n-\lambda|\int |T|^2 + 2\int |T|^2 |\nabla \varphi|^2. \]
        Sending $\varepsilon\to 0$ and using that $T\in L^\infty(\Omega)$, we conclude that $|\nabla T|\in L^2(\Omega)$. 

        Next, by \eqref{eq.equation for Delta |T| square}, we have that
        \[\int \varphi^2 \frac12 \Delta |T|^2 = \int \varphi^2 |\nabla T|^2  +(2n-\lambda)\int \varphi^2 |T|^2.\]
        Since $\varphi$ is supported on the regular part of  $\overline \Omega$ and $|\varphi|\le 1$, we have:
        \[ \left| \int \varphi^2 \Delta |T|^2 \right| = 4\left| \int |T| \langle \nabla |T|,\nabla \varphi\rangle \right| \le  4 \|T\|_{L^\infty} \|\nabla |T|\|_{L^2} \|\nabla \varphi\|_{L^2} \to 0,\]
        as $\varepsilon\to 0$. In the last step we used the Kato inequality $|\nabla |T||\le |\nabla T|$ so $\| \nabla |T|\|_{L^2(\Omega)} <\infty$.
    \end{proof}
    
    Next, we compare $|\nabla T|^2$ and $|\Div T|^2$ for a traceless tensor $T$.

    \begin{lemma}\label{lemma.algebraic inequality}
        Let $A$ be a $(0,3)$-tensor in $\R^n$ such that for all $i, j, k$,
        \[A_{ijk} = A_{jik}, \quad \sum_i A_{iik} = 0.\]
        Then 
        \[\sum_{i,j,k} |A_{ijk}|^2 \ge \frac{2n}{(n-1)(n+2)} \sum_i \left(\sum_j A_{ijj}\right)^2.\]
        Moreover, equality holds if and only if there exists a vector $v\in \R^n$ such that
			\[A_{ijk} = \delta_{jk} v_i + \delta_{ik} v_j - \frac2n \delta_{ij} v_k. \]
    \end{lemma}

    \begin{proof}
    
        Denote by $b_i = \sum_j A_{ijj}$ the $i$-th component of the right hand side. We introduce the following auxiliary tensor $\Phi$:
        \[\Phi_{ijk} =\frac12 (\delta_{jk}b_i + \delta_{ik}b_j ) - \frac1n \delta_{ij}b_k.\]
        We observe that $\Phi$ and $A$ has the same symmetries: obviously we have that $\Phi_{ijk} = \Phi_{jik}$, and:
        \[\sum_i \Phi_{iik} = \sum_{i\ne k} - \frac 1n b_k + b_k  -\frac 1n b_k = 0.\]
        Suppose now that $B$ is an arbitrary $(0,3)$-tensor admitting the same symmetry, that is, for all $i, j, k>0$, $B_{ijk} = B_{jik}$ and $\sum_i B_{iik} = 0$. We compute:
        \begin{equation}\label{eq. B pair Phi}
            \begin{split}
                \langle B, \Phi \rangle &= \sum_{i,j,k} B_{ijk} \left(\frac12 \delta_{jk}b_i + \delta_{ik} b_j - \frac 1n \delta_{ij}b_k  \right)\\
            &= \sum_{i,j}\left(\frac12 B_{ijj}b_i + \frac12 B_{iji} b_j -\frac1n B_{iij}b_j\right)\\
            &=\sum_{i,j} B_{ijj} b_i.
            \end{split}
        \end{equation}
        In particular,
        \begin{equation}\label{eq. A pair Phi}
            \langle A, \Phi \rangle = \sum_i b_i^2=|b|^2.
        \end{equation}
        We next compute $|\Phi|^2$. For this, we observe:
        \[\sum_j \Phi_{ijj} = \frac{n}{2} b_i + \frac12 b_i - \frac 1n b_i = \frac{(n-1)(n+2)}{2n}b_i. \]
        Thus by \eqref{eq. B pair Phi},
        \begin{equation}\label{eq.|Phi| square}
            |\Phi|^2 =\frac{(n-1)(n+2)}{2n} |b|^2.
        \end{equation}

        Now set $C = A - \frac{2n}{(n-1)(n+2)}\Phi$. We see that $C$ has the additional symmetry that $\sum_j C_{ijj} = 0$. Therefore, using \eqref{eq. B pair Phi} again we have that $\langle C, \Phi\rangle = 0$. Consequently, 
		\[|A|^2 = \left|C+ \frac{2n}{(n-1)(n+2)} \Phi\right|^2 = |C|^2 + \left|\frac{2n}{(n-1)(n+2)} \Phi\right|^2. \]
        The inequality follows. Moreover, equality holds if and only if $C = 0$, so $A = \frac{2n}{(n-1)(n+2)}\Phi$, as desired.
    \end{proof}

    To finish the proof, we apply Lemma \ref{lemma.algebraic inequality} to $A_{ijk} = (\nabla_k T)(\partial_i,\partial_j)$, and obtain that
    \[|\nabla T|^2 \ge \frac{2n}{(n-1)(n+2)} |\Div T|^2. \]
    Therefore
    \begin{align*}
        0 & = (2n-\lambda)\int |T|^2 + |\nabla T|^2 \\
          & \ge (2n-\lambda)\cdot \frac{n-1}{n}(\lambda -n) \int |\nabla u|^2 + \frac{2n}{(n-1)(n+2)} \int |\Div T|^2 \\
          & = \frac{(2n-\lambda)(n-1)}{n} (\lambda-n) \int |\nabla u|^2 \\
          &\qquad + \frac{2n}{(n-1)(n+2)} \left(\frac{n-1}{n} (\lambda-n)\right)^2 \int |\nabla u|^2.
    \end{align*}
    Since $u$ is nonconstant and $\lambda>n$, solving the above gives $\lambda \ge 2(n+1)$. This finishes the proof of the Neumann spectral gap.

    \vspace{1em}

It is well-known that any function $u$ on a subset of $S^n$ with $\Delta u = -n u$ is the restriction of a linear function in $\R^{n+1}$. We now characterize all eigenfunctions in the next mode.

Suppose $u\in C^2(\overline{\Omega})$ satisfies that
    \[\begin{cases}
        \Delta u = -2(n+1) u &\quad \text{ in }\Omega,\\
        \partial_\nu u = 0 &\quad \text{ on }\partial \Omega.
    \end{cases}.\]

    Set $T = \nabla^2 u + \tfrac{2(n+1)}{n} u g$ the traceless Hessian and $\bar u(x) = |x|^2 u(\tfrac{x}{|x|})$ the quadratic homogeneous extension of $u$ in the cone of $\Omega\subset \R^{n+1}$. It follows that $\bar u$ is harmonic. Fix a point in the interior of $\Omega$ and consider normal coordinates $\{x^1,\cdots, x^n\}$. Set $r=|x|$, $e_i = \partial_i$ and $e_0 = \partial_r$ the normal vector. Let $\bar \nabla$ be the Euclidean connection. Extend $\{e_i\}_{i=1}^n$ to an orthonormal frame in an open subset of $\R^{n+1}$.

    Tracking the proof of Theorem \ref{theorem.eigenvalue estimate}, we deduce that equality in Lemma \ref{lemma.algebraic inequality} holds for $A_{ijk} = (\nabla_k T)(e_i,e_j)$ pointwise. Thus there exists a tangent vector $v$ such that
    \[(\nabla_k T)(e_i,e_j) = \delta_{jk}v_i + \delta_{ik} v_j - \frac2n \delta_{ij}v_k. \]
    By Proposition \ref{prop. geometric equations for T}, $\sum_k \nabla_k(\nabla^2 u)_{kj} = \nabla_j (\Delta u) + (n-1) u_j = -2(n+1) u_j + (n-1)u_j$. Therefore $\sum_k (\nabla_k T)(e_k, e_j) = - \frac{(n+2)(n-1)}{n} u_j$. Comparison gives that $v= -\nabla u$.  Consequently, we have that
    \[(\nabla_k T)(e_i,e_j) = -\delta_{jk} u_i - \delta_{ik} u_j + \frac2n \delta_{ij}u_k,\]
    or equivalently, 
    \begin{equation}\label{eq. equality.for.third.derivative.u}
        (\nabla_k (\nabla^2 u))(e_i,e_j) + 2\delta_{ij} u_k + \delta_{jk} u_i + \delta_{ik} u_j = 0.
    \end{equation}
    We now deduce that \eqref{eq. equality.for.third.derivative.u} is precisely equivalent to that all third derivatives of $\bar u$ vanishes pointwise in the interior of the cone over $\Omega$. This would finish the proof.

    We proceed with computing $\bar \nabla_0 \bar u = 2ru$, so $\bar \nabla^2 \bar u (e_0,e_0) = 2 u$ and for $i, j>0$,
    \[\bar \nabla^2 \bar u(e_i, e_0) = \partial_i (2u) - du (\bar \nabla_{e_i} e_0 ) = u_i, \]
    \begin{multline*}
        \bar \nabla^2 u(e_i,e_j) = \nabla^2 u(e_i,e_j) -d\bar u (\bar \nabla_{e_i}e_j) = \nabla^2 u(e_i,e_j) -d\bar u(-\delta_{ij}e_0) \\
        = \nabla^2 u(e_i,e_j) + 2\delta_{ij} u.
    \end{multline*}
    Consequently, along tangential directions $i, j, k>0$, 
    \begin{align*}
        \bar \nabla^3_{e_i, e_j, e_k} \bar u&= \bar\nabla_{e_i} (\bar \nabla^2 \bar u) (e_j, e_k)\\
        &= e_i \left( (\bar \nabla^2 u) (e_j, e_k)\right) -\bar\nabla^2 \bar u(\bar \nabla_{e_i}e_j, e_k) - \bar \nabla^2 \bar u(e_j, \bar \nabla_{e_i}e_k)\\
        &=e_i \left( \nabla^2 u(e_j, e_k) + 2u \delta_{jk} \right) + \bar\nabla^2 \bar u(\delta_{ij} e_0, e_k) + \bar\nabla^2 \bar u (e_j, \delta_{ik} e_0)\\
        &=(\nabla e_i (\nabla^2 u))(e_j,e_k) + 2\delta_{kj} u_i + \delta_{ij} u_k +\delta_{ik} u_j = 0,
    \end{align*}
    by \eqref{eq. equality.for.third.derivative.u}. The fact that $\bar \nabla^3 \bar u = 0$ in directions containing $e_0$ follows from the fact that $\bar u$ is quadratic homogeneous.

\section{The implication $(P_n)\Rightarrow (Q_{n+1})$}

In this section we deduce the implication $(P_n) \Rightarrow (Q_{n+1})$.
Take $\Omega \subset \R^{n+1}$ is a polyhedral cone domain.  We say $\Omega$ is $l$-symmetric if (up to rigid motion) we can express $\Omega = \Omega_0 \times \R^l$, but \emph{cannot} write $\Omega = \Omega_0' \times \R^{l+1}$.  If $\Omega = \Omega_0 \times \R^l$ is $l$-symmetric, it will be convenient to write $x = (z, y) \in \R^{n-l} \times \R^l$.

Fix $\Omega$ a polyhedral cone domain in $\R^{n+1}$.  Since the link $\Omega \cap S^n$ is a Lipschitz domain, there is an $L^2(\Omega \cap S^n)$-orthonormal basis of Neumann eigenfunctions $\phi_i$ with associated eigenvalues $0 = \mu_0 < \mu_1 \leq \mu_2 < \ldots \to \infty$, so that each $\phi_i$ solves
\[
\int_{\Omega \cap S^n} \nabla \phi_i \cdot \nabla \zeta = \mu_i \int_{\Omega \cap S^n} \phi_i \zeta \quad \forall \zeta \in C^1(S^n).
\]
Following \eqref{eq. expansion.rate.vs.eigenvalue} it will be convenient to define $\gamma_i = - ((n-1)/2) + \sqrt{ ((n-1)/2)^2 + \mu_i}$, which are the admissible powers of homogeneous, locally-finite-energy harmonic functions in $\Omega$ with Neumann boundary data.

It is a classical fact (see e.g. \cite{Mazya2009boundednessofgradient,EscobarEigenvalueEstimate}), since $\Omega$ is convex, that if $\mu_i > 0$ then $\mu_i \geq n$, and the only eigenfunctions corresponding to $\mu_i = n$ are restrictions of linear functions.  This is equivalent to the condition that any $(\gamma > 0)$-homogeneous Neumann harmonic function in $\Omega$ must be linear, and is the key to $C^{1,\alpha}$ regularity for the Neumann problem in $\Omega$. We rely on the next order eigenvalue estimate supplied by Property $(P_n)$.

To prove regularity for polyhedral cone domains $\Omega \subset \R^n$ we will induct on the dimension of symmetry of $\Omega$.  It will be useful to distinguish properties $(P_n)$, $(Q_n)$ for a specific domain:

\begin{property}[Property $P_\Omega$]
Write $0 = \mu_0 < \mu_1 \leq \mu_2 < \cdots$ for the Neumann eigenvalues of the link $\Omega \cap S^n$, and $\phi_i$ for the associated Neumann eigenfunctions.

If $\Omega$ is $0$-symmetric, then $\mu_1 \geq 2(n+1)$.  If $\Omega = \Omega_0 \times \R^l$ is $l$-symmetric, then $\mu_1 = \ldots = \mu_l = n$ and $\phi_1, \ldots, \phi_l$ are restrictions to $\Omega \cap S^n$ of the linear functions on $\{0\} \times\R^l$, and $\mu_{l+1} \geq 2(n+1)$.  In either case, if any eigenfunction $\phi_i$ has eigenvalue $\mu_i = 2(n+1)$, then $\phi_i$ is the restriction of a homogeneous quadratic polynomial to $\Omega \cap S^n$.
\end{property}
An equivalent characterization of $(P_\Omega)$ is that any $(1+\alpha > 1)$-homogeneous solution of \eqref{eqn:eqn} with locally-finite energy is at least $2$-homogeneous, and can only be $2$-homogeneous if quadratic

We will show property $(P_\Omega)$ holds for any polyhedral cone domain with dihedral angle $\leq \pi/2$.  For this, we have seen a regularity estimate (property $(Q_\Omega)$ below) is necessary.

\begin{property}[Property $Q_\Omega$]
Let $u \in W^{1,2}(\Omega \cap B_1)$, $f \in C^\alpha(\Omega \cap B_1)$ for $\alpha \in (0, 1)$ solve 
\[
\Delta u = f \text{ in } \Omega \cap B_1, \quad D_\nu u = 0 \text{ in } \del \Omega \cap B_1
\]
in a distributional sense, i.e. so that
\[
\int_{\Omega \cap B_1} Du \cdot D\zeta = - \int f \zeta \quad \forall \zeta \in C^1_c(B_1).
\]
Then there is a $\beta(\Omega, \alpha) \in (0, 1)$ so that $u \in C^{2,\beta}(\Omega \cap B_1)$, and we have the estimate
\[
|u|_{C^{2,\beta}(\Omega \cap B_{1/2})} \leq c(\Omega, \alpha)( ||u||_{L^2(\Omega \cap B_1)} + |f|_{C^\alpha(\Omega \cap B_1)}).
\]
\end{property}

We now carry out the inductive step. We start by pointing out that the higher-order eigenvalue estimate of $(P_\Omega)$ is ultimately a property of $0$-symmetric domains.

\begin{lemma}\label{lem:add-cross}
Suppose we can write $\Omega = \Omega_0 \times \R^l$ where $\Omega_0$ is $0$-symmetric.  If statement $(P_{\Omega_0})$ holds, then $(P_\Omega)$ holds also, in which case any quadratic solution $u$ to \eqref{eqn:eqn} in $\Omega$ must split as
\begin{equation}\label{eqn:add-cross-concl}
u(z, y) = q_{ij} z_i z_j + a_{ij} y_i y_j.
\end{equation}
\end{lemma}

\begin{proof}
Let $\{ \gamma_{i,0} \}_i$, $\{ \gamma_{i} \}_i$ be the homogeneities (not counting multiplicity) of possible homogeneous, locally-finite-energy solutions to \eqref{eqn:eqn} in $\Omega_0$, $\Omega$ respectively.  It holds that $\gamma_{0, 0} = \gamma_0 = 0$, $\gamma_{1, 0} = \gamma_1 = 1$, and by our assumption we have $\gamma_{2, 0} = 2$.

Let $u \in W^{1,2}(\Omega \cap B_1)$ be a $\gamma$-homogeneous solution to \eqref{eqn:eqn} in $\Omega$.  We first note that necessarily $\gamma \geq 0$.  We second observe that any derivative $D_{y_i} u$ is a $(\gamma-1)$-homogeneous solution to \eqref{eqn:eqn} in $\Omega$, and also lies in $W^{1,2}(\Omega \cap B_1)$.  Therefore we must have $D^k_y u = 0$ for $k > \gamma$, and so $u$ takes the form
\[
u(z, y) = \sum_{|J|=0}^{\lfloor \gamma \rfloor} b_J(z) y^J
\]
where each $b_J(z)$ is $(\gamma-|J|)$-homogeneous.

If $\gamma \in (1, 2)$, then $u(z, y) = b_0(z) + b_{1, i}(z) y_i$ where $b_0$ is $(1+\alpha)$-homogeneous and $b_{1, i}(z)$ is $\alpha$-homogeneous.  We have
\[
0 = \Delta u = (\Delta_z b_0) + (\Delta_z b_{1,i}) y_i
\]
which implies each $b_0, b_{1, i}$ solve \eqref{eqn:eqn} in $\Omega_0$, but this contradicts the fact that neither $\alpha$ nor $1+\alpha$ is an admissible homogeneity in $\{ \gamma_{i, 0} \}_i$.  We deduce $\gamma_2 \geq 2$.

If $\gamma = 2$, then $u(z, y) = b_0(z) + b_{1, i}(z) y_i + b_{2, ij}(z) y_i y_j$, where $b_0$ is $2$-homogeneous, $b_{1, i}$ is $1$-homogeneous, and $b_{0, ij}(z)$ is $0$-homogeneous.  We have
\[
0 = \Delta u = (\Delta_z b_0 + 2 b_{2, ii}) + (\Delta_z b_{1, i}) y_i + (\Delta_z b_{2, ij}) y_i y_j,
\]
which implies $b_{2, ij}$, $b_{1, i}$, and $b_0 - \frac{1}{n+1+l} b_{2ii} |z|^2$ solve \eqref{eqn:eqn} in $\Omega_0$.

Since $(P_{\Omega_0})$ holds we must have $b_0(z) = q_{ij} z_i z_j$.  Since $\Omega_0$ is convex we must have $b_{1, i}(z) = p_{ij} z_j$ and $b_{2, ij}(z) = a_{ij} \in \R$, but since $\Omega_0$ is $0$-symmetric then in fact $b_{1, i} = 0$.  So $u$ takes the form \eqref{eqn:add-cross-concl}, and we deduce $(P_\Omega)$ holds.
\end{proof}

\subsection{Regularity estimate}

In this section $\Omega$ will always denote a polyhedral cone domain.  We are interested in $u \in W^{1,2}(\Omega \cap B_1)$ solving the Poincare equation with zero Neumann data, i.e. $u$ solving
\begin{equation}\label{eqn:eqn}
\int_{\Omega \cap B_1} Du \cdot D\zeta = - \int_{\Omega \cap B_1} f \zeta \quad \forall \zeta \in C^1_c(B_1),
\end{equation}
for some $f \in L^\infty(\Omega \cap B_1)$ or $f \in C^\alpha(\Omega \cap B_1)$.

We require the following classical regularity result.
\begin{theorem}[\cite{NazarovPlamenevsky+1994}]\label{thm:C1alpha-reg}
Let $u \in W^{1,2}(\Omega \cap B_1)$ solve \eqref{eqn:eqn} with $f \in L^\infty(\Omega \cap B_1)$.  Then there there is an $\alpha(\Omega) \in (0, 1)$ so that $u \in C^{1,\alpha}(\Omega \cap B_1)$, and
\[
|u|_{C^{1,\alpha}(\Omega \cap B_{1/2})} \leq c(\Omega) (||u||_{L^2(\Omega \cap B_1)} + ||f||_{L^\infty(\Omega \cap B_1)}).
\]
\end{theorem}

Define the norm
\begin{align*}
|||f|||_{m, \Omega}^2 &= \int_0^1 r^{-1-2m} ||f(r \cdot)||^2_{L^2(\Omega \cap \del B_r)} dr .
\end{align*}

\begin{lemma}[Existence of inhomogeneous solutions]\label{lem:inhomog-soln}
Take $m \not \in \{ \gamma_i \}_i$, and $f \in L^2(\Omega \cap B_1)$ with $|||f|||_{m, \Omega} < \infty$.  Then we can find a $u \in W^{1,2}(\Omega \cap B_1)$ solving
\[
\Delta u = f \text{ in } \Omega \cap B_1, \quad D_\nu u = 0 \text{ on } \del\Omega \cap B_1
\]
in the distributional sense (i.e. so that \eqref{eqn:eqn} holds) and admitting the bound
\begin{equation}\label{eqn:inhomog-concl2}
    |||u|||_{m+2, \Omega} \leq c(\Omega, m) |||f|||_{m, \Omega} .    
\end{equation}

In particular we have
\[
r^{-(n+1)} \int_{\Omega \cap B_r} u^2 \leq r^{2(m+2)} |||u|||_{m+2, \Omega}^2 \leq c(\Omega, m) r^{2(m+2)} |||f|||_{m, \Omega}^2 \quad \forall 0 < r < 1.
\]
\end{lemma}

\begin{proof}
We seek a solution to \eqref{eqn:eqn} of the form
\begin{equation}\label{eqn:inhomog-soln-1}
u = \sum_{i=0}^\infty a_i(r) \phi_i(\theta).
\end{equation}
Writing
\[
f_i(r) = \int_{\Omega \cap \del B_1} f(r \theta) \phi_i(\theta) d\theta
\]
then the $a_i(r)$ must solve (distributionally)
\begin{align}
f_i(r) &= a_i''(r) + nr^{-1} a_i(r) - r^{-2}  \mu_i a_i(r) \label{eqn:a-ode}\\
&\equiv r^{-n-\gamma_i}( r^{n+2\gamma_i} (r^{-\gamma_i} a(r))')' . \nonumber
\end{align}

Let us define
\[
a_i(r) = r^{\gamma_i} \int_{r_i}^r s^{-n-2\gamma_i} \left( \int_0^s t^{n+\gamma_i} f_i(t) dt \right) ds
\]
for $r_i = 1$ if $\gamma_i > 2+m$ and $r_i = 0$ if $\gamma_i < 2+m$.  By a similar computation as in \cite[Lemma 4.1]{edelen2026entireareaminimizingsurfacesmathbfrn}, we get
\begin{gather*}
|a_i(r)|^2 + \frac{r^2}{(\gamma_i+1)^2} |a'_i(r)|^2 \leq \frac{c(\Omega, m)}{(\gamma_i+1)^3}  r^{2(2+m)} \int_0^1 t^{-1-2m} f_i(t)^2 dt 
\end{gather*}
for every $0 < r < 1$, from which it follows
\begin{equation}\label{eqn:inhomog-soln-2}
\sum_{i=0}^\infty (1+\mu_i) |a_i(r)|^2 + r^{2} |a_i'(r)|^2 \leq c(\Omega, m) r^{2(2+m)} |||f|||_{m, \Omega}^2 \quad \forall 0  < r < 1.
\end{equation}
Using \eqref{eqn:inhomog-soln-2} and the orthogonality of the Neumann eigenfunctions $\phi_i$, it's straightforward to show that
\begin{align*}
\int_{\Omega \cap B_1} \left| \sum_{i=0}^N a_i \phi_i \right|^2 + \left| \sum_{i=0}^N D(a_i \phi_i) \right|^2  dx
&= \int_{\Omega \cap B_1} \sum_{i=0}^N (a_i \phi_i)^2 + \sum_{i=0}^N |D(a_i \phi_i)|^2 dx \\
&\leq \int_0^1 \left( |a_i'(r)|^2 + (1+r^{-2} \mu_i) |a_i(r)|^2 \right) r^{n} dr \\
&\leq c(\Omega, m) |||f|||_{m, \Omega}^2
\end{align*}
for all natural numbers $N$, and hence $u_N(x = r\theta) := \sum_{i=0}^N a_i(r) \phi_i(\theta)$ conveges in $W^{1,2}(\Omega \cap B_1)$ to some function $u$ satisfying the estimate \eqref{eqn:inhomog-concl2}.

On the other hand, if we let $f_N = \sum_{i=0}^N f_i(r) \phi_i(\theta)$, then from the definition of $a_i$, each $u_N$ solves
\[
\int_\Omega D u_N \cdot D\zeta = - \int_\Omega f_N \zeta \quad \forall \zeta \in C^1_c(B_1)  ,
\]
which from our aforestated convergence implies $u$ solves \eqref{eqn:eqn}.
\end{proof}

The property $(P_\Omega)$ implies the following radial decay estimate.
\begin{lemma}[centered $C^{2,\beta}$ decay]\label{lem:radial-decay}
Let $u \in W^{1,2}(\Omega \cap B_1)$ solve \eqref{eqn:eqn} with $f \in C^\alpha(\Omega \cap B_1)$ for $\alpha \in (0, 1)$, and suppose $(P_\Omega)$ holds.  Suppose also that $u(0) = 0$, $Du|_0 = 0$.

Then there is a $\beta(\Omega, \alpha) \in (0, 1)$ and a homogenous degree-2 polynomial $q(x) = q_{ij} x_i x_j$ solving \eqref{eqn:eqn} with $f(0)$ in place of $f$ so that
\begin{equation}\label{eqn:radial-decay-concl1}
|u - q|_{C^0(\Omega \cap B_r)} \leq c(\Omega, \alpha) r^{2+\beta} (||u||_{L^2(\Omega \cap B_1)} + |f|_{C^\alpha(\Omega \cap B_1)}) \quad \forall 0 < r < 1/2,
\end{equation}
and
\begin{equation}\label{eqn:radial-decay-concl2}
|q_{ij}| \leq c(\Omega, \alpha) (||u||_{L^2(\Omega \cap B_1)} + |f|_{C^\alpha(\Omega \cap B_1)}).
\end{equation}
In particular, $u$ is twice-differentiable at $0$ with $D^2_{ij} u|_0 = q_{ij}$.
\end{lemma}

\begin{proof}
Set $v = u - \frac{1}{2(n+1)} f(0) |x|^2$, and then $v$ (distributionally) solves \eqref{eqn:eqn} with $f - f(0)$ in place of $f$.  By Theorem \ref{thm:C1alpha-reg} and our assumptions on $u$ and $f$, after shrinking $\alpha(\Omega)$ as necessary, we have the bounds
\begin{gather}\label{eqn:radial-decay-1}
|v|_{C^0(\Omega \cap B_r)} \leq r^{1+\alpha} C \quad \forall 0 < r < 1/2, \\
\text{ and } |f - f(0)|_{C^0(\Omega \cap B_r)} \leq r^\alpha [f]_{\alpha, \Omega \cap B_{1}} \quad \forall 0 < r < 1 ,
\end{gather}
for some constant $C$ independent of $r$.

Fix $0 < m < \alpha/2$ so that $m \not \in \{ \gamma_i \}_i$, and apply Lemma \ref{lem:inhomog-soln} to find a $v_0 \in W^{1,2}(\Omega \cap B_1)$ solving \eqref{eqn:eqn} with $f - f(0)$ in place of $f$, and admitting the bound
\begin{align}
r^{-(n+1)} \int_{\Omega \cap B_r} v_0^2 
&\leq c r^{2(m+2)} |||f - f(0)|||_{m, \Omega} \nonumber \\
& \leq c(\Omega, \alpha) r^{2(m+2)} [f]_{\alpha, \Omega \cap B_1} \quad \forall 0 < r < 1. \label{eqn:radial-decay-1.5}
\end{align}
It then follows from Theorem \ref{thm:C1alpha-reg} that
\begin{equation}\label{eqn:radial-decay-2}
|v_0|_{C^0(\Omega \cap B_r)} \leq c(\Omega, \alpha) r^{m+2} [f]_{\alpha, \Omega \cap B_1} \quad \forall 0 < r < 1/2.
\end{equation}

Set $w = v - v_0$.  Then $w \in W^{1,2}(\Omega \cap B_1)$ (distributionally) solves \eqref{eqn:eqn} with $0$ in place of $f$, and by \eqref{eqn:radial-decay-1}, \eqref{eqn:radial-decay-2} admits the bound
\begin{equation}\label{eqn:radial-decay-3}
|w|_{C^0(\Omega \cap B_r)} \leq r^{1+\alpha} C' \quad \forall 0 < r < 1/2.
\end{equation}
Define
\begin{equation}\label{eqn:radial-decay-3.5}
a_i(r) = \int_{\Omega \cap \del B_1} w(r \theta) \phi_i(\theta) d\theta,
\end{equation}
and then using the ODE \eqref{eqn:a-ode}, the bound $|a_i(r)| \leq c(n) C' r^{1+\alpha}$, and the property $(P_\Omega)$ we deduce
\[
a_i(r) = 0 \text{ if } \gamma_i < 2, \quad a_i(r) = A_i r^{\gamma_i} \text{ if } \gamma_i \geq 2 \text{ for } A_i \in \R.
\]

Using again property $(P_\Omega)$, we can therefore expand in $W^{1,2}(\Omega \cap B_1)$
\begin{align*}
w(x = r \theta) 
&= \sum_{\gamma_i = 2} A_i r^2 \phi_i(\theta) + \sum_{\gamma_i \geq 2+\beta'} A_i r^{\gamma_i} \phi_i(\theta) \\
&= a_{ij} x_i x_j + \sum_{\gamma_i \geq 2+ \beta'} A_i r^{\gamma_i} \phi_i(\theta) 
\end{align*}
for some $\beta'(\Omega) > 0$.  We deduce
\begin{align}
r^{-(n+1)} \int_{\Omega \cap B_r} |w - a_{ij} x_i x_j|^2 
&= r^{-(n+1)} \int_0^r \sum_{\gamma_i \geq 2+\beta'} A_i^2 s^{2\gamma_i + n} ds \nonumber \\
&\leq r^{2(2+\beta')} \int_{\Omega \cap B_1} w^2 \label{eqn:radial-decay-4}
\end{align}
for every $r < 1$.  If we define $q_{ij} = a_{ij} + \frac{1}{2(n+1)} f(0) \delta_{ij}$ and $\beta = \min \{ \beta', m \}$, then by combining Theorem \ref{thm:C1alpha-reg}, \eqref{eqn:radial-decay-4}, \eqref{eqn:radial-decay-1.5} we have
\begin{align*}
|u - q_{ij} x_i x_j|_{C^0(\Omega \cap B_r)}
&\leq |w - a_{ij} x_i x_j|_{C^0(\Omega \cap B_r)} + |v_0|_{C^0(\Omega \cap B_r)} \\
&\leq c r^{2+\beta'} ||w||_{L^2(\Omega \cap B_1)} + c r^{2+m} [f]_{\alpha, \Omega \cap B_1} \\
&\leq c(\Omega, \alpha) r^{2 + \beta} ( ||u||_{L^2(\Omega \cap B_1)} + |f|_{C^\alpha(\Omega \cap B_1)}).
\end{align*}
which is the desired conclusion \eqref{eqn:radial-decay-concl1}.  The bound \eqref{eqn:radial-decay-concl2} on $q_{ij}$ follows directly from \eqref{eqn:radial-decay-3.5}, \eqref{eqn:radial-decay-1.5}.
\end{proof}

Given $\Omega = \Omega_0 \times \R^l$ which is $l$-symmetric, and given $\chi > 0$, define the wedge-type domain
\[
\Omega_{\chi, y} = \{ x = (z, y) \in \Omega_0 \times \R^l : |y - y'| \leq \chi |z| \}.
\]

Let us note that for any $x = (z, y) \in \Omega \setminus \{0 \} \times \R^l$, there is a radius $r(x) \equiv r(z/|z|)$ so that $\Omega \cap B_{r |z|}(x)$ is an $l'$-symmetric polyhedral domain (centered at $x$) for some $l' > l$.  In particular, for any $\chi >0$, one can find a finite collection of such balls so that the collection $\{ B_{r_i(x_i)/2}(x_i) \}_i$ covers $\Omega_{\chi, 0} \cap \{ 1/2 \leq |z| \leq 1 \}$, and the balls $\{ B_{r_i}(x_i)\}_i$ lie in $\Omega_{2\chi, 0} \cap \{ 1/4 \leq |z| \leq 2 \}$.

\begin{lemma}\label{thm:recenter-decay}
Let $u \in W^{1,2}(\Omega \cap B_1)$ solve \eqref{eqn:eqn} with $f \in C^\alpha(\Omega \cap B_1)$.  Suppose $(P_\Omega)$.  Suppose additionally that $\Omega = \Omega_0 \times \R^l$ is $l$-symmetric, and $(Q_{\Omega'})$ holds for all polyhedral cone domains $\Omega' \subset \R^{n+1}$ which are $l'$-symmetric for $l' > l$.\footnote{Really it suffices to assume $(Q_{\Omega'})$ holds for every tangent cone $\Omega' = T_x \Omega$, for $x \not\in \{0\}\times \R^l$} Fix $\chi \geq 1$.

Then there is a $\beta(\Omega, \alpha) > 0$ so that for every $y \in B^l_{1/2}$, we can find a degree $2$-polynomial $q|_y(x)= q_{ij}(y) (x_i - y_i)(x_j - y_j)$ solving \eqref{eqn:eqn} with $f(y)$ in place of $f$, so that
\begin{equation}\label{eqn:recenter-decay-concl1}
|q_{ij}(y)| \leq c(\Omega) (||u||_{L^2(\Omega \cap B_1)} + |f|_{C^\alpha(\Omega \cap B_1)}),
\end{equation}
and for all $0 < r < 1/8$ we have the estimates
\begin{equation}\label{eqn:recenter-decay-concl2}
|D^2_{ij} u - q_{ij}(y)|_{C^0(\Omega_{\chi, y} \cap B_r(y))} \leq c(\Omega, \chi) r^\beta ( ||u||_{L^2(\Omega \cap B_1)} + |f|_{C^\alpha(\Omega \cap B_1)})
\end{equation}
and
\begin{equation}\label{eqn:recenter-decay-concl2.5}
[D^2_{ij} u]_{\beta, \Omega_{\chi, y} \cap \{ r/2 \leq |z| \leq r \}} \leq c(\Omega, \chi) ( ||u||_{L^2(\Omega \cap B_1)} + |f|_{C^\alpha(\Omega \cap B_1)}).
\end{equation}
In particular, $y \in B^l_{1/2} \mapsto q_{ij}(y)$ is $C^\beta$, with the estimate
\begin{equation}\label{eqn:recenter-decay-concl3}
|q_{ij}|_{C^\beta(B_{1/2}^l)} \leq c(\Omega) (||u||_{L^2(\Omega \cap B_1)} + |f|_{C^\alpha(\Omega \cap B_1)}).
\end{equation}
\end{lemma}

\begin{proof}
For ease of notation write $\Lambda = ||u||_{L^2(\Omega \cap B_1)} + |f|_{C^\alpha(\Omega \cap B_1)}$.  Let $\{ B_{r_i}(x_i) : x_i = (z_i, y_i)\}_i$ be the finite cover from above.  Fix $y' \in B^l_{1/2}$.  We know from Theorem \ref{thm:C1alpha-reg} that $u$ is $C^{1, \alpha}$, and so we can apply Lemma \ref{lem:radial-decay} in $B_{1/2}(0, y')$ to the function $v(z, y) = u - u(0, y') - \sum_{i=1}^l (y_i - y'_i) D_{y_i} u|_{(0, y')}$ to deduce the existence of $q|_{y'}(x) = q_{ij}(y')(x_i - y'_i)(x_j - y'_j)$ which solve \eqref{eqn:eqn} with $f(y')$ in place of $f$, and satisfy
\begin{align*}
|q_{ij}(y')| &\leq c ||v||_{L^2(\Omega \cap B_{1/2}(0, y'))} + c |f|_{C^\alpha(\Omega \cap B_{1/2}(0, y'))}\\
&\leq c ||u||_{L^2(\Omega \cap B_1)} + c |u(y')| + c|Du|_{(0, y')}| + c |f|_{C^\alpha(\Omega \cap B_1)} \\
&\leq c(\Omega) \Lambda,
\end{align*}
and 
\[
|v - q|_{y'}|_{C^0(B_r(0, y'))} \leq c(\Omega) r^{2+\beta} \Lambda \quad \forall r \leq 1/4.
\]

Noting that $v - q|_{y'}$ solves \eqref{eqn:eqn} with $f - f(y')$ in place of $f$, for any radius $r \in (0, 1/8]$, we can apply Property $Q$ (after shrinking $\beta$ as necessary) to the function $v - q|_{y'}$ in the balls $B_{r_i r}(y' + rx_i)$ to deduce the estimate
\begin{align*}
&r^2 |D^2 (v - q|_{y'})|_{C^0(\Omega_{\chi, y'} \cap \{ r/2 \leq z \leq r \})} + r^{2+\beta} [D^2(v - q|_{y'})]_{\beta, \Omega_{\chi, y'} \cap \{ r/2 \leq |z| \leq r \}} \\
&\leq c ||v - q_{y'}||_{L^2(\Omega_{2\chi, y'} \cap \{ r/4 \leq z \leq 2r \})} + c r^2|f - f(y')|_{C^0(\Omega_{2\chi, y'} \cap \{ r/4 \leq z \leq 2r \})} \\
&\quad + c r^{2+\beta} [f]_{\beta, \Omega_{2\chi, y'} \cap \{ r/4 \leq z \leq 2r \}} \\
&\leq c(\Omega, \chi) r^{2+\beta} \Lambda
\end{align*}
Since $D^2_{ij}(v - q|_{y'}) = D^2_{ij} u - q_{ij}(y')$, we get \eqref{eqn:recenter-decay-concl2}.

It remains to prove \eqref{eqn:recenter-decay-concl3}.  Take $y', y'' \in B_{1/2}^l$ with $d := |y' - y''| \leq 1/1000$.  Then since $\chi \geq 1$ we have
\begin{align*}
|\Omega_{\chi, y'} \cap \Omega_{\chi, y''} \cap B_{10 d}( (y' + y'')/2)| 
&\geq |\Omega \cap \{ 4d \leq |z| \leq 8d \} \cap \{ |y - (y' + y'')/2| \leq d \}| \\
&\geq d^{n+1} /c(\Omega).
\end{align*}
Using the above and \eqref{eqn:recenter-decay-concl2}, we can compute
\begin{align*}
&d^{n+1} |q_{ij}(y') - q_{ij}(y'')|^2/c(\Omega) \\
&\leq 2 \int_{\Omega_{\chi, y'} \cap \Omega_{\chi, y''} \cap B_{10 d}((y'+y'')/2)} |q_{ij}(y') - D^2 u_{ij}(x)|^2 + |q_{ij}(y'') - D^2_{ij}u(x)|^2 dx \\
&\leq 2 \int_{\Omega_{\chi, y'} \cap B_{20d}(y')} |q_{ij}(y') - D^2_{ij} u(x)|^2 dx + 2 \int_{\Omega_{\chi, y''} \cap B_{20d}(y'')} |q_{ij}(y'') - D^2_{ij}u(x)|^2 dx \\
&\leq c(\Omega) d^{n+1+2\beta} \Lambda.
\end{align*}
This completes the proof of Lemma \ref{thm:recenter-decay}.
\end{proof}

We are now ready to prove our inductive step.
\begin{lemma}\label{thm:induct-reg}
Let $\Omega = \Omega_0 \times \R^l$ be $l$-symmetric.  Suppose $(P_\Omega)$ holds, and $(Q_{\Omega'})$ holds for every $l'$-symmetric polyhedral cone domain $\Omega'$ for $l' > l$.  Then $(Q_\Omega)$ holds also.
\end{lemma}

\begin{proof}
Write $\Lambda = ||u||_{L^2(\Omega \cap B_1)} + |f|_{C^\alpha(\Omega \cap B_1)}$.  It already follows from Lemma \ref{thm:recenter-decay} that $u \in C^2(\overline{\Omega} \cap B_{1/2})$, and $|u|_{C^2}(\Omega \cap B_{1/2}) \leq c(\Omega) \Lambda$.  We now show the $C^{2,\beta}$ estimate.

Take $x, x' \in B_{1/2} \cap \Omega$, and write $x = (z, y)$, $x' = (z', y')$.  Fix any $\chi \geq 1$.  First suppose $x \not \in \Omega_{\chi, y'}$ and $x' \not\in \Omega_{\chi, y}$.  Then we have
\[
|y - y'| > \max \{ \chi |z|, \chi |z'| \},
\]
and so we can use \eqref{eqn:recenter-decay-concl2}, \eqref{eqn:recenter-decay-concl3} to get
\begin{align*}
|D^2_{ij}u(x) - D^2_{ij} u(x')|
&\leq |D^2_{ij} u(x) - q_{ij}(y)| + |q_{ij}(y) - q_{ij}(y')| + |q_{ij}(y') - D^2_{ij}u(x')| \\
&\leq (c |z|^\beta + c |y - y'|^\beta + c |z'|^\beta) \Lambda \\
&\leq c|y - y'|^\beta \Lambda \\
&\leq c(\Omega, \chi)|x - x'|^\beta \Lambda.
\end{align*}

Suppose now $x \in \Omega_{\chi, y'}$.  First we can use \eqref{eqn:recenter-decay-concl2.5} to get
\[
|D^2_{ij} u(z, y) - D^2_{ij} u(z, y')| \leq c(\Omega, \chi) |y - y'|^\beta \Lambda,
\]
so we are left with bounding
\begin{equation}\label{eqn:induct-reg-1}
|D^2_{ij} u(z, y') - D^2 u_{ij} u(z', y')| \leq c(\Omega, \chi) |z - z'|^\beta \Lambda.
\end{equation}
Up to relabeling, we can suppose that $0 < |z'| \leq |z|$, and now break into two cases.  If $| |z'| - |z|| \leq |z|/10$, then we can again use \eqref{eqn:recenter-decay-concl2.5} to get \eqref{eqn:induct-reg-1}.  Otherwise, we have $|z'| \leq |z| \leq 10 ||z'| - |z|| \leq 10 |z - z'|$, and we can use \eqref{eqn:recenter-decay-concl2} to bound
\begin{align*}
|D^2_{ij} u(z, y') - D^2 u_{Ij}(z', y')| 
&\leq |D^2_{ij} u(z, y') - q_{ij}(y')| + |D^2_{ij} u(z', y') - q_{ij}(y')| \\
&\leq (c |z|^\beta + c |z'|^\beta) \Lambda \\
&\leq c(\Omega, \chi)|z - z'|^\beta \Lambda.
\end{align*}
The case when $x' \in \Omega_{\chi, y}$ is verbatim.  This proves Theorem \ref{thm:induct-reg}.
\end{proof}

\begin{theorem}\label{thm:P-implies-Q}
Suppose $(P_\Omega)$ holds for every polyhedral cone domain $\Omega \subset \R^{n+1}$.  Then $(Q_{\Omega \times \R^l})$ holds for any polyhedral cone domain $\Omega \subset \R^{n+1}$ and every $l \geq 0$. Equivalently, $(Q_{\Omega'})$ holds for any polyhedral cone domain $\Omega' \subset \R^{n+1+l}$ which is $(l' \geq l)$-symmetric.
\end{theorem}

\begin{proof}
First note by Lemma \ref{lem:add-cross}, we get $(P_{\Omega \times \R^l})$ for every such $\Omega \subset \R^{n+1}$, $l \geq 0$.  Fix an $l$, and now we prove by induction that $(Q_{\Omega'})$ holds for any $(l' \geq l)$-symmetric polyhedral cone domain $\Omega' \subset \R^{n+1+l}$.  If $l' = n+1+l$ or $l' = n+l$ then $(Q_{\Omega'})$ trivially holds by standard interior or boundary estimates.  By inductive hypotheses we can therefore assume $(Q_{\Omega'})$ holds for $(l'' > l')$-symmetric domains $\Omega'' \subset \R^{n+1+l}$, but then Lemma \ref{thm:induct-reg} implies $(Q_{\Omega'})$ holds also, and hence Theorem \ref{thm:P-implies-Q} follows by  induction.
\end{proof}

\bibliographystyle{plain}
\bibliography{bib.bib}

\end{document}